\documentclass[12pt]{amsart}
 \usepackage{graphicx}
\usepackage{amsmath,amsthm,amsfonts,amscd,amssymb,mathrsfs,hyperref}
\usepackage[usenames,dvipsnames]{xcolor}
\newtheorem{theorem}{Theorem}[section]

\newtheorem{prop}[theorem]{Proposition}

\newtheorem{computation}[theorem]{Computation}
\newtheorem{algor0}{Algorithm}

\def\rig#1{\smash{ \mathop{\longrightarrow}
 \limits^{#1}}}

\newcommand{\Gr}{\operatorname{Gr}}
\newcommand{\SL}{\operatorname{SL}}
\newcommand{\SO}{\operatorname{SO}}

\newcommand{\V}{{\mathcal V}}

\newcommand{\rank}{\operatorname{rank}}
\newcommand{\Frank}{\operatorname{F-Rank}}
\newcommand{\Prank}{\operatorname{P-Rank}}

\newcommand{\codim}{\operatorname{codim}}
\newcommand{\PP}{\mathbb{P}}
\newcommand{\RR}{\mathbb{R}}
\newcommand{\CC}{\mathbb{C}}

\newcommand{\Seg}{\operatorname{Seg}}
\newcommand{\Sub}{\operatorname{Sub}}
\newcommand{\Rank}{\operatorname{Rank}}
\def\bw#1{{\textstyle \bigwedge^{\hspace{-.2em}#1}}}

\theoremstyle{definition}

\theoremstyle{remark}
\newtheorem{remark}[theorem]{Remark}

\begin{document}

\author{Chris Aholt$^{a}$, Luke Oeding$^{b}$}
\address{$^{A}$ Department of Mathematics, University of Washington, Seattle, WA, USA}
\address{$^{B}$ Department of Mathematics,
University of California Berkeley,
Berkeley, CA, USA
}
\email{aholtc@uw.edu$^{A}$}
\email{oeding@math.berkeley.edu $^{B}$}

\date{\today}

\newcommand{\red}[1]{{\color{red} \sf #1}}

\title{The ideal of the trifocal variety}
\begin{abstract}
Techniques from representation theory, symbolic computational algebra, and numerical algebraic geometry are used to find the minimal generators of the ideal of the trifocal variety. An effective test for determining whether a given tensor is a trifocal tensor is also given.
\end{abstract}
\maketitle

\section{Introduction}

In the field of multiview geometry one studies $n\ge 2$ planar images of points in space.
Given $n$ full rank $3\times 4$ matrices
$A_1,\ldots,A_n$ over $\mathbb{C}$, these {\em camera matrices} determine a rational map
\[\phi: \PP^3\dashrightarrow (\PP^2)^n\qquad x\mapsto (A_1x,\cdots, A_nx)\]
from projective 3-space into the $n$-fold product of projective planes.
For any given tuple $(A_1,\ldots,A_n)$ the image of this map determines
a variety $\overline{\phi(\PP^3)}\subseteq(\PP^2)^n$ called
the {\em multiview variety associated to $(A_1,\ldots,A_n)$}.

In \cite{AST} the authors determine the prime ideal
defining the multiview variety for a generic fixed tuple of cameras such that the camera matrices $A_{1},\dots,A_{n}$ have pairwise distinct kernels.
In this paper we focus on a different, but related variety in the special case of $n=3$ cameras:
the variety of all trifocal tensors \cite[Ch. 15]{HZ}. 
The essential difference is that for the multiview variety the camera matrices are fixed and this determines a map from the world to a set of images, but in the trifocal setup we consider the set of tensors determined by all possible general configurations of triples of cameras.
Algebraically, the collection of trifocal tensors is parameterized by the $4\times 4$ minors of the $4\times 9$ matrix $(A_1^T \mid A_2^T\mid A_3^T)$ which involve one column each from the first two blocks, and two columns from the third block \cite{Heyden}. Geometrically, a trifocal tensor arises from a bilinear map describing the geometry of a given configuration of cameras.   We give a complete description of the ideal describing this subvariety of tensors.

We further describe this geometric map. Each camera matrix $A_i$ determines a {\em focal point} $f_i\in\PP^3$ and a {\em viewing plane} $\pi_i\simeq \PP^2\subseteq\PP^3$.  The image in camera $i$ of a point $x\in\PP^3$ is determined by intersecting the line $\langle f_i,x\rangle$ with the plane $\pi_i$.  Now consider lines $l_i\subseteq \pi_i$ for $i=1,2$.  The planes $\langle f_1,l_1\rangle$ and $\langle f_2,l_2\rangle$ generically intersect in a line $l_{1,2}\subseteq\PP^3$, and the plane $\langle f_3,l_{1,2}\rangle$ generically intersects $\pi_3$ in a line $l_3$.  See Figure~\ref{fig:trifocal 3d}.

We have described for a sufficiently general camera configuration, a map
\[
\PP^2 \times \PP^2 \rightarrow \PP^2
,\]
given by $(l_1,l_2)\mapsto l_3$.  This map must come from a bilinear map
\[
 \CC^3 \times  \CC^3 \rightarrow  \CC^3
.\]
To help avoid ambiguity, fix $A,B,C\simeq \CC^3$ so that this map is now $A\times B\to C$.
This bilinear map is equivalently a tensor $T \in A^{*}\otimes B^{*}\otimes C$, called a {\em trifocal tensor} because of its derivation. For more details, see \cite[Chapter 15]{HZ},\cite{Alzati-Tortora}, \cite{Heyden}, \cite{PF}.
\begin{figure}\caption{The trifocal tensor as a map $\PP^2 \times \PP^2 \rightarrow \PP^2$}\label{fig:trifocal 3d}
\vspace{0.2cm}
\includegraphics[width=0.7\linewidth]{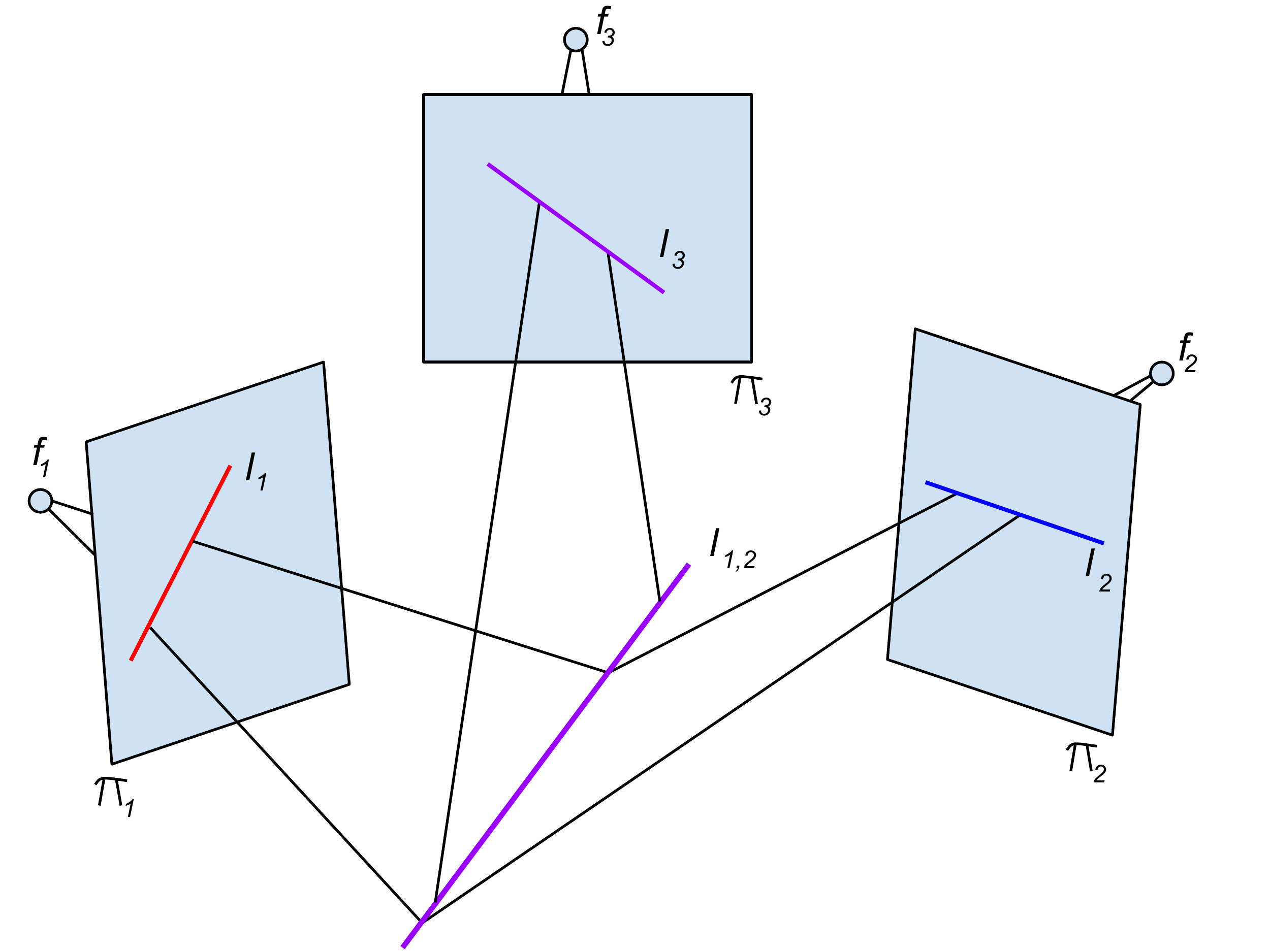}
\end{figure}


One way to connect these two seemingly different algebraic and geometric constructions is via the following construction, which shows (in an invariant way) how the parametrization using special minors of a $4\times 9$ matrix gives rise to a tensor parametrization.  This also relates to the compactified camera space considered in \cite{AST}.

The row space of $\left( A_1^T\mid A_2^T\mid A_3^T\right)$ determines a point in the Grassmannian $\Gr(4,9)$.
Set $U_{1},U_{2},U_{3}$ respectively as the 3-dimensional column spaces of the blocked matrix $\left( A_1^T\mid A_2^T\mid A_3^T\right)$. The direct sum $W = U_{1}\oplus U_{2} \oplus U_{3}$ is 9-dimensional, and we can view the matrix  $\left( A_1^T\mid A_2^T\mid A_3^T\right)$ as describing a point in the Grassmannian
\[
\Gr(4,W) \subset \PP (\bw4 W)
.\]
Consider the group $G=\SL(U_{1})\times \SL(U_{2}) \times SL(U_{3}) \subset SL(W)$, which can be thought of as the group of (unit determinant) $3\times3$ blocks on the diagonal of a $9\times 9$ matrix. Now $\bw4 W$ decomposes as a $G$-module as follows
\begin{multline*}
\bw{4}(U_{1}\oplus U_{2} \oplus U_{3}) = \left(\bigoplus_{i\neq j}U_{i}\otimes \bw3 U_{j} \right) \oplus\left(\bigoplus_{i\neq j}\bw2U_{i}\otimes \bw2 U_{j} \right)
\\
\oplus\left(\bigoplus_{i,j,k\ \text{distinct}}U_{i}\otimes  U_{j}\otimes \bw2U_{k} \right)
.\end{multline*}
If we take $A^{*}=U_{1}, B^{*}=U_{2}$ and $C^{*}=U_{3}$, we see that the factor $U_{1}\otimes  U_{2}\otimes \bw2U_{3}$ is isomorphic to $A^{*}\otimes B^{*} \otimes C$ and corresponds to the space of maximal minors of a $4\times 9$ matrix using 1 column from the first two $4\times 3$ blocks and 2 columns from the last $4\times 3$ block.
So we get a $G$-equivariant projection 
\[\pi :\Gr(4,9) \dashrightarrow \PP \left(U_{1}\otimes  U_{2}\otimes \bw2U_{3}\right) \cong \PP (A^{*}\otimes B^{*}\otimes C),\]  and the closure of the image of this projection is the trifocal variety $X$. Because the projection is $G$-equivariant, the image $X$ is automatically $G$-invariant.  
The generic fiber of the projection $\pi$ is $3$ dimensional. However, 
there is a $(\CC^*)^{3}$ action by scaling each $A_i$ that leaves the set of trifocal tensors invariant. The GIT quotient $\Gr(4,9)/\!/(\CC^*)^{3}$ has dimension 18, which is the dimension of $X$, \cite{AST,Alzati-Tortora}.
So
if one works with the GIT quotient, the map
\[\pi :\Gr(4,9)/\!/(\CC^{*})^{3} \dashrightarrow \PP \left(U_{1}\otimes  U_{2}\otimes \bw2U_{3}\right),\] 
 becomes birational to $X$. 

One would like to know when a given tensor in $V=A^{*}\otimes B^{*}\otimes C$ arose as a trifocal tensor.  The Zariski closure of the set of all such trifocal tensors defines an irreducible algebraic variety, called the \emph{trifocal variety} and hereafter denoted by $X\subset \PP V$. Let $I(X)$ denote the ideal of polynomial functions vanishing on $X$, hereafter called the \emph{trifocal ideal}.  Since a tensor $T$ is a trifocal tensor (or a limit of such) if and only if $T$ is a zero of every polynomial in the ideal $I(X)$, the question of identifying trifocal tensors can be answered (at least for general tensors in an open set of $X$) by determining the minimal generators of $I(X)$.

In \cite{Alzati-Tortora} the authors determine a set of polynomials that cut
out $X$ as a set.  However, their set of polynomials does not
generate the ideal $I(X)$. We note that \cite{PF} and \cite{Ressl} also found some equations vanishing on $X$, but neither described the entire trifocal ideal.
The focus of this article is to determine the minimal generators of $I(X)$.

Choosing bases $\{a_{1},a_{2},a_{3}\}$, $\{b_{1},b_{2},b_{3}\}$, and $\{c_{1},c_{2},c_{3}\}$ of $A^{*}$, $B^{*}$ and $C$ respectively,
any tensor $T\in V$ can be realized as
\[T = \sum_{i,j,k=1}^{3} T_{i,j,k}a_{i}\otimes b_{j} \otimes c_{k}\]
via the 27 variables $T_{i,j,k}$ for $1\le i,j,k\le 3$
Therefore, the trifocal ideal lives in the polynomial ring $k[T_{ijk}]$.

The cubic polynomials in the ideal are the 10 coefficients of
$$\det(x_1 T_{ij1} + x_2 T_{ij2} + x_3 T_{ij3}).$$
One component of the zero set of these polynomials is our variety.  To remove the other components we add polynomials of degree 5 and 6.

To simplify matters, we will take the ground field to be $k=\CC$; however, we note that in practice, one works over $\RR$.  A tensor with real entries is on the complex trifocal variety if and only if it is a zero of all polynomials in $I(X)$.  And indeed, all of the generating polynomials in $I(X)$ can be taken with rational coefficients, and thus are in the ideal of $X$ when considered as a variety over $\RR$. 

Our result is the following.
\begin{theorem}\label{thm:main} Let $X$ denote the trifocal variety.
The prime ideal $I(X)$ is minimally generated by 10 polynomials in degree 3, 81 polynomials in degree 5, and 1980 polynomials in degree 6.
\end{theorem}

There are noticeably more generators here than in \cite{Alzati-Tortora}, which showed that 10 equations of degree 3,
20 of degree 9, and 6 of degree 12 cut out $X$ set-theoretically. On the other hand, the degrees of our equations are lower and we know that they are the minimal degree polynomials that generate the ideal.

\section{Outline}
To prove Theorem \ref{thm:main} and determine the minimal generators of the trifocal ideal, we use a mixture of several different computational and theoretical tools that we now outline. In short, our strategy is to first find equations in the ideal in the lowest degrees, next show that the equations we found cut out the variety set theoretically and thus define an ideal that agrees up to radical with the ideal we want, and then we show that the two ideals are actually equal.

Because the trifocal construction is unchanged by changes of coordinates in each camera plane, we have a large group $G$ that acts on $X$. We describe this symmetry and various representations for points on $X$ in Section \ref{sec:trifocal}.  Then we describe the geometry of related $G$-varieties in $\PP V=\PP(A^{*}\otimes B^{*}\otimes C)$ in Section \ref{sec:other varieties}.

We let $I(X)_{d}$ denote the degree $d$ piece of $I(X)$, and denote by 
 $M_{d}(X)$, (or $M_{d}$ when the context is clear), the (vector space of) minimal generators of $I(X)$ occurring in degree $d$. The group $G$ acts on $X$ and thus on $I(X)$. This facilitates the search for all polynomials in $I(X)_{d}$  in low degree (for $d\leq 9$) via computations using classical representation theory.
In Section \ref{sec:rep} we describe our computations and identify which modules of polynomials are minimal generators assisted by symbolic computations in Maple and Macaulay2. In particular, we find that the only minimal generators of $I(X)$ for $d\leq 9$ occur in degrees 3, 5 and 6.
Next, we compute a Gr\"obner basis of the ideal  $J=\langle M_{3}+M_{5}+M_{6} \rangle$, and find (again in Macaulay2) that the degree of $J$ is 297.

Another valuable tool for understanding the zero-set of a collection of polynomials is a numerical primary decomposition via numerical algebraic geometry. 
For this we used Bertini \cite{Bertini} for experiments and computations; see also \cite{SW05}.
In Section \ref{sec:Bertini} we consider only $M_{3}$, the lowest degree (degree 3) part of $I(X)$, which has a basis of 10 polynomials in the 27 variables.  Here we obtain a numerical primary decomposition of $V(M_{3})$ using Bertini.
In particular, we find that up to the numerical accuracy of Bertini, $V(M_{3})$ has 4 components, and we are even given their degrees. This numerical result provides us with tangible data from which we are able to conjecture (and eventually prove) the true structure of $V(M_{3})$.
In particular, we find that up to the numerical accuracy of Bertini, $X$ has degree 297.

In Section \ref{sec:Nurmiev} we use geometric considerations and resort to Nurmiev's classification of orbits and their closures (\cite{Nurmiev, Nurmiev2}) to geometrically identify all the components found by the Bertini computation. This geometric understanding allows us to conclude in Proposition \ref{prop:irred} that the zero-set $\mathcal{V}(J)$ is equal to $X$ (as sets), so $\sqrt{J} = I(X)$. In Section \ref{sec:steven} we again use the classification of orbits and the orbit poset structure to show in Theorem \ref{thm:prime} that $J$ is prime and thus $J = I(X)$.

\section{The trifocal variety as an orbit closure}\label{sec:trifocal}

Consider $V=A^{*}\otimes B^{*}\otimes C$ and the natural left action of $G =\SL(A)\times \SL(B)\times \SL(C)\simeq \SL(3)^{\times 3}$ on $V$.  There is also a natural action of the symmetric group $\mathfrak{S}_{3}$ permuting the three factors in the tensor product, and it is easy to see that $X$ is invariant under the action of the $\mathfrak{S}_{2}$ permuting $A^{*}$ and $B^{*}$. However, this finite invariance does not provide much computational advantage.

\begin{remark}
Since we are working over $\CC$ we consider general changes of coordinates by $\SL(3,\CC)$. However, were we to work over $\RR$, we would want to change our analysis to consider rotations in the three planes, and the group action would be by $\SO(3,\RR)^{\times 3}$.
\end{remark}
Since the trifocal variety $X\subset \PP V$ is invariant under changes of coordinates in the camera planes, we say that $X$ is a $G$-variety.
Moreover, \cite{Alzati-Tortora} shows that $X$ is actually the closure of a single $G$-orbit in $\PP V$.

Because $G\simeq \SL(3)^{\times 3}$ is $24$ dimensional acting on $V\simeq\CC^{3}\otimes \CC^{3}\otimes \CC^{3}$, which is 27 dimensional, there must be infinitely many $G$-orbits in $V$.  Even so, the orbits happen to have been classified, apparently independently, by several authors.
Since elements of $V$ can be interpreted in a number of ways (as triples of $3\times 3$ matrices or $3\times 3$ matrices with linear entries depending on $3$ variables,  as $3-3-3$ trilinear forms or ternary trilinear forms, as cuboids or elements of a triple tensor product, or as a $G$-submodule of $ \bw{3}\CC^{9}$), the various classifications occurred in different settings --- see \cite{Thrall-Chanler, Ng, Nurmiev}.
 
We prefer to use Nurmiev's version of the classification, which follows Vinberg's conventions and uses the results and techniques of \cite{Vinberg-Elasvili}. One main reason for this choice is that Nurmiev also computed the closures of all the nilpotent orbits in a note \cite{Nurmiev2}, in the same language as the previous paper. 
There are 4 continuous families of orbits called semi-simple orbits, and one finite family of nilpotent orbits.  To every orbit $\mathcal{O}$ is associated a \emph{normal form}, which is a representative $v\in V$ such that $G.v = \mathcal{O}$. Though obviously not unique, we will typically choose a normal form that is as simple as possible, or that clearly reveals some of the structure of the orbit.

To use the Nurmiev classification, we first identify the trifocal variety as one of the orbits on Nurmiev's list.
Indeed, Alzati and Tortora give a normal form for the orbit of trifocal tensors that we now recall.  A general trifocal tensor $T$ may be, after a possible change of coordinates by $G$, identified as a tensor whose slices in the $C$-direction are
\[
T^{1} = \begin{pmatrix}
0&-1&0 \\
0&0&0 \\
1&0&0
\end{pmatrix}
T^{2} = \begin{pmatrix}
0&0&0 \\
0&-1&0 \\
0&1&0
\end{pmatrix}
T^{3} = \begin{pmatrix}
0&0&0 \\
0&0&0 \\
0&-1&1
\end{pmatrix}
\]
Choose bases $A^* = {\rm span}\{a_{1},a_{2},a_{3}\}$, $B^* = {\rm span}\{b_{1},b_{2},b_{3}\}$, and $C = {\rm span}\{c_{1},c_{2},c_{3}\}$.  Then we can write
\[T=
(-a_{1})\otimes b_{2}\otimes c_{1} + (-a_{3}) \otimes(- b_{1})\otimes c_{1} 
+(-a_{2})\otimes b_{2}\otimes c_{2} \]\[+ a_{3}\otimes b_{2}\otimes c_{2}
-a_{3}\otimes b_{2}\otimes c_{3} + a_{3}\otimes b_{3}\otimes c_{3}
,\]
which by changing coordinates (via $G$) may be written as
\[
T=
a_{1}\otimes b_{2}\otimes c_{1} + a_{3}\otimes b_{1}\otimes c_{1} 
+ a_{2}\otimes b_{2}\otimes c_{2}+ a_{3}\otimes b_{3}\otimes c_{3}
.\]

It is also useful to express a tensor $T$ via matrices with linear entries.  For this, one considers a pure tensor $a_{i}\otimes b_{j}\otimes c_{k}$ as a matrix with an $a_{i}$ in the $j,k$ position of the matrix.  Then do this for all pure tensors in an expression for $T$ and add the matrices. In fact, this describes the projections for the $\Prank$ varieties defined in Section~\ref{sec:other varieties}. A normal form for the trifocal variety has matrices of linear forms
\[
T(A) = 
 \begin{pmatrix}
a_{3}&0&0 \\
a_{1}&a_{2}&0 \\
0&0&a_{3} 
\end{pmatrix},\ 
T(B) = 
 \begin{pmatrix}
b_{2}&0&0 \\
0&b_{2}&0 \\
b_{1}&0&b_{3} 
\end{pmatrix},\ 
T(C) = 
 \begin{pmatrix}
0&c_{1}&0 \\
0&c_{2}&0 \\
c_{1}&0&c_{3} 
\end{pmatrix}
.\]

The difference here is that $T(A)$, $T(B)$, $T(C)$ each individually represent $T$, but the entire set $\{T^{1},T^{2},T^{3}\}$ also represents $T$.  One advantage to considering a tensor as a matrix in linear forms is that it is now clear that $\Prank(T) =(3,3,2)$, so $X \subset \Prank^{3,3,2}$. In particular, $T(C)$ has rank 2, and thus must satisfy the equations implied by $\det(T(C)) \equiv 0$, while $T(A)$ and $T(B)$ do not satisfy this relation.

\begin{remark}The construction of the matrix $T(A)$ from the tensor $T$ shows that the $G$ action on $T$ corresponds to an action on the matrix $T(A)$ by left and right multiplication by elements of $\SL(B)$ and $\SL(C)$, and by linear changes of variables on the entries of $T(A)$, with similar descriptions for the action on $T(B)$ and $T(C)$.
\end{remark}

Nurmiev lists the $G$-orbits in $V$ as a list of integers. 
To a triple of integers $ijk$ Nurmiev associates the tensor $e_{i}\otimes e_{j}\otimes e_{k}$, with $1\leq i, j-3, k-6 \leq 3$. The spaces in each expression corresponds to summation.

For example, consider orbit $11$ on Nurmiev's list: $149\; 167\; 248\; 357$. 
We choose bases of $A^{*}$, $B^{*}$ and $C$ so that $a_{i}=e_{i}$, $b_{j-3} = e_{j}$, and $c_{k-6}=e_{k}$.  So orbit 11 corresponds to the tensor
\[
a_{1}\otimes b_{1}\otimes c_{3}+ a_{1}\otimes b_{3}\otimes c_{1}+ a_{2}\otimes b_{1}\otimes c_{2}+ a_{3}\otimes b_{2}\otimes c_{1},
\]
which corresponds to the matrix of linear forms
\[
\begin{pmatrix}
0 & a_{2} & a_{1}\\
a_{3} & 0 & 0 \\
a_{1} & 0 & 0
\end{pmatrix}
.\]

Finally, notice that $T(C)$, for instance, can be moved by a change of coordinates to
\[T(C) = 
 \begin{pmatrix}
0&c_{1}&0 \\
0&c_{2}&0 \\
c_{1}&0&c_{3} 
\end{pmatrix}
\cong
 \begin{pmatrix}
0&0&c_{1} \\
0&0&c_{2}\\
c_{1}&c_{3}&0 
\end{pmatrix}
\cong
 \begin{pmatrix}
0&c_{1}&c_{3}\\
c_{1}&0&0 \\
c_{2}&0&0
\end{pmatrix}
\cong 
\begin{pmatrix}
0&c_{3}&c_{1}\\
c_{2}&0&0 \\
c_{1}&0&0
\end{pmatrix}
.\]
Then swapping the roles of $c_{2}$ and $c_{3}$ we obtain
\[T(C)
\cong
\begin{pmatrix}
0&c_{2}&c_{1}\\
c_{3}&0&0 \\
c_{1}&0&0
\end{pmatrix}
.\]
This shows that the normal form of $T$ is congruent to orbit $11''$ on Nurmiev's list (where representative $11''$ is obtained from representative $11$ by performing the permutation $a\to b\to c\to a$ twice).

Nurmiev's list also contains the dimension of the stabilizer of this orbit.  This confirms for us that the codimension of the trifocal variety $X$ is 8 (an already well-established fact).

One of the utilities of having a group action is the following. If a group $G\subset GL(V)$ preserves a variety $X\subset \PP V$ (i.e. $G.X =X$), we may consider $I(X)$ as a $G$-module and in particular as a $G$-submodule of $S^{\bullet}V^{*},$ the space of symmetric tensors on $V$.  Recall that $S^{\bullet}V^{*}$ is isomorphic to the space of homogeneous polynomials on $V$.  The representation theory of  $G = \SL(A)\times \SL(B)\times \SL(C)$-modules is well known; however, the reader may wish to consult \cite{LandsbergTensorBook} or \cite{FultonHarris} for reference.  One fact we will use is if $V=A^{*}\otimes B^{*}\otimes C$, then irreducible $G$-modules in $S^{\bullet}V^{*}$ are all of the form $S_{\lambda}A\otimes S_{\mu}B\otimes S_{\nu}C^{*}$, where $\lambda,\mu$ and $\nu$ are all partitions of the same nonnegative integer.

\section{$\Frank$ and $\Prank$ varieties}\label{sec:other varieties}

In the previous section we saw that the matrices $T^{1},T^{2},T^{3}$ representing the slices of a trifocal tensor do not have full rank. This condition depends on the choice of coordinates. On the other hand, the $3\times 9$ flattening matrix $( T^{1}\mid T^{2}\mid T^{3})$ does have full rank, and this condition is not dependent on the choice of coordinates. Of course slicing in a different direction may yield a different result, but it is easy to check that trifocal tensors have full rank flattenings for all slices.   We refer to this condition as \emph{flattening rank ($\Frank$)}, and note that the general trifocal tensor has $\Frank(T)=(3,3,3)$.

The matrix $T(C)$ with linear forms in $C$ does not have full rank, while the matrices $T(A)$ and $T(B)$ do have full rank.  The construction of $T(C)$ describes a projection $A^{*}\otimes B^{*}\otimes C \to A^{*}\otimes B^{*}$,
so it is natural to refer to the tuple of ranks of the various projections as \emph{projection rank} ($\Prank$).  A general trifocal tensor $T$ has $\Prank(T)=(3,3,2)$.   

These two considerations lead to the study of subspace varieties (the former) and rank varieties (the latter).  Understanding algebraic and geometric properties of these varieties will help us find equations for the trifocal variety. In what follows we highlight some of these properties, which are specific cases of much more general constructions. For more details, see \cite[Chapter 7]{LandsbergTensorBook}).

\subsection{Subspace varieties}
For $p\leq 3$, $q\leq 3$, $r\leq 3$, the \emph{subspace variety} $\Sub_{p,q,r} \subset \PP V$ is the projectivization of the set of tensors that have $\Frank$ at most $(p,q,r)$:
\[\Sub_{p,q,r} = \PP\{T\in V \mid \Frank(T) \leq (p,q,r)  \}
,\]
where we write $(a,b,c) \leq (p,q,r)$ if $a\leq p$ and $b\leq q$ and $c\leq r$.
Subspace varieties are irreducible, and their ideals are defined by minors of flattenings (see \cite[Theorem 3.1]{Landsberg-Weyman-Bull07}).  For the sake of a reader unfamiliar with these concepts, we recall the construction of these equations.

A tensor $T\in V$ is realized via 27 variables $T_{i,j,k}$ for $1\le i,j,k\le 3$
\[T = \sum_{i,j,k=1}^{3} T_{i,j,k}a_{i}\otimes b_{j} \otimes c_{k},\]
choosing bases $\{a_{1},a_{2},a_{3}\}$, $\{b_{1},b_{2},b_{3}\}$, and $\{c_{1},c_{2},c_{3}\}$ of $A^{*}$, $B^{*}$ and $C$ respectively.
There are three directions in which we may slice $T$ to get triples of matrices.
Let $W_{i} = (T_{i,j,k})_{j,k}$, $Y_{j} = (T_{i,j,k})_{i,k}$, $Z_{k} = (T_{i,j,k})_{i,j}$ denote these slices.
Then the matrices $W=(W_{1}\mid W_{2}\mid W_{3})$ --- respectively $Y=(Y_{1}\mid Y_{2}\mid Y_{3})$,  and $Z=(Z_{1}\mid Z_{2}\mid Z_{3})$ --- are the three flattenings with respect to the three slicings of the tensor $T$.

A special case of \cite[Theorem 3.1]{Landsberg-Weyman-Bull07} is that the ideals generated by the 3-minors of flattenings are the ideals of subspace varieties. Namely,
\[
I(\Sub_{2,3,3} ) =\langle minors(3,W) \rangle,
\]
\[I(\Sub_{3,2,3} ) =\langle minors(3,Y) \rangle,
\]
\[I(\Sub_{3,3,2} ) =\langle minors(3,Z)\rangle.
\]
Moreover, the intersection of two subspace varieties yields another, and this holds ideal theoretically as well.
$\Sub_{2,2,3} = \Sub_{2,3,3} \cap \Sub_{3,2,3}$ and
\[
I(\Sub_{2,2,3} ) =\langle minors(3,W) \rangle + \langle minors(3,Y) \rangle,
\]
and similarly for permutations. 
Likewise 
\[
I(\Sub_{2,2,2} ) =\langle minors(3,W) \rangle +\langle minors(3,Y) \rangle + \langle minors(3,Z)\rangle
.\]
It is also easy to check the dimensions of subspace varieties. A convenient tool is to use the Kempf-Weyman desingularization via vector bundles.  Let $\mathcal{S}_{i}$ denote the canonical (subspace) rank $i$ vector bundle over the Grassmannian $\Gr(i,n)$.  The desingularization is 
\[
\PP(\mathcal{S}_{p}\otimes \mathcal{S}_{q}\otimes \mathcal{S}_{r})
\times \Gr(p,3)\times \Gr(q,3) \times \Gr(r,3) \dashrightarrow \Sub_{p,q,r}
.\]
In particular, we have
\[
\dim(\Sub_{p,q,r}) = pqr-1 + p(3-p) + q(3-q) +r(3-r)
.\]
We computed the degrees of each subspace variety using Macaulay 2:
\begin{center}
\begin{tabular}{lccc}
\text{variety}&$\Sub_{2,3,3}$ & $\Sub_{2,2,3}$ & $\Sub_{2,2,2}$ \\
\text{dimension}&19 & 15  &13   \\
\text{codimension}&7 & 11  &13   \\
\text{degree}&36 & 306  & 783   \\
\end{tabular}
\end{center}

Another description of $I(\Sub_{p,q,r})$ in Representation Theoretic language will allow us to compare the equations here with any other $G$-invariant sets of equations, no matter how they are presented. Another way to state \cite[Theorem 3.1]{Landsberg-Weyman-Bull07} is that for each integer $d$, $I(\Sub_{p,q,r})_{d}$ consists of all representations $S_{\lambda}A \otimes S_{\mu} B \otimes S_{\nu} C^{*}$ with partitions $\lambda, \mu, \nu \vdash d$ such that either $|\lambda|>p$, $|\mu| >q$ or $|\nu| >r$.

In fact, the ideals of subspace varieties are generated in the minimal degree $d$ possible. To save space, we write $S_{\lambda}S_{\mu}S_{\nu}$ for the representation $S_{\lambda}A\otimes S_{\mu}B\otimes S_{\nu}C^{*}$, and $\bw{3}$ in place of $S_{111}$
\[I(\Sub_{2,3,3}) =\langle
 \bw{3}  \bw{3} S^{3} \oplus
 \bw{3} S_{21} S_{21} \oplus
 \bw{3} S^{3}  \bw{3}
\rangle
\]
\[ I(\Sub_{223}) =I(\Sub_{233}) + \langle 
 S_{21}  \bw{3} S_{21} \oplus
 S^{3}  \bw{3}  \bw{3} 
\rangle\]
\[ I(\Sub_{2,2,2}) = I(\Sub_{2,2,3})+ \langle 
S_{21} S_{21}  \bw{3}
\rangle
\]

Finally, comparing to Nurmiev's list \cite{Nurmiev2}, the variety $\Sub_{2,3,3}$ corresponds to nilpotent orbit 9 (and $9'$ and $9''$ correspond to permutations of $\Sub_{2,3,3}$). The variety $\Sub_{2,2,3}$ corresponds to nilpotent orbit number 17 (and $17'$ and $17''$ for permutations).
$\Sub_{2,2,2}$ is also equal to the secant variety of a Segre product, $\sigma_{2}({\rm \Seg}(\PP^{2}\times \PP^{2}\times \PP^{2}))$ and corresponds to nilpotent orbit number 20 on Nurmiev's list.

\subsection{$\Prank$ varieties}
$\Prank$ varieties 
are defined by considering the three images of the projections of a tensor onto two of the factors and restricting the rank of points in the image. In particular (see \cite[\S7.2.2]{LandsbergTensorBook}) $\Rank^{r}_{A} $ is the projectivization of the set
\[ \{T \in V \mid \rank(T(A))\leq r\}
,\]
where if $T$ has slices (in the $A^{*}$-direction) $W_{1}$, $W_{2}$ and $W_{3}$,  $T(A)$ is the matrix $W_{1}+W_{2}+W_{3}$. $\Rank^{r}_{B}$ and $\Rank^{r}_{C}$ are defined similarly.  Let $\Prank^{p,q,r}$ denote the projectivization of the set of tensors $T$ with $\Prank(T) \leq (p,q,r)$.  Equivalently,  
\[\Prank^{p,q,r} =\Rank^{p}_{A} \cap \Rank^{q}_{B} \cap \Rank^{r}_{C} .\]
It is easy to check that $\Prank^{p,q,r}$ is $\SL(A)\times \SL(B)\times \SL(C)$-invariant.

While $\Prank$ varieties can be considered in arbitrary dimensions, we restrict to the case that $A$, $B$ and $C$ are 3-dimensional.

Landsberg points out that $\Rank^{r}_{A}$ is usually far from irreducible. 
In particular, there are at least two subspace varieties in $\Rank^{2}_{A}$
\[
\Sub_{3,3,2} \cup \Sub_{3,2,3} \subset \Rank^{2}_{A} = \Prank^{2,3,3}
.\]
In fact, we will see later that there is yet another component.
Similarly 
\[
\Sub_{3,3,2} \cup \Sub_{3,2,3} \subset  \Prank^{2,3,3},
\]
\[
\Sub_{3,3,2} \cup \Sub_{2,3,3} \subset  \Prank^{3,2,3}
\]
imply a third containment $\Sub_{3,3,2}\cup \Sub_{2,2,3} \subset \Prank^{2,2,3}$.

Moreover the 3-way intersection certainly contains the following
 \[\Sub_{3,2,2} \cup \Sub_{2,3,2}  \cup \Sub_{2,2,3}  \subset \Prank^{2,2,2}.\]
But, in fact, all of the inclusions above are strict containments.

In Section \ref{sec:Nurmiev} we consider the poset of orbit closures in $\Prank^{3,3,2}$.  This will allow us to show that $\Prank^{2,2,2}$ is irreducible and corresponds to the orbit closure consisting of tensors whose projections $T(A)$, $T(B)$ and $T(C)$ are skew-symmetrizable $3\times 3$ matrices.
One can also show that the irreducible components of $\Prank^{3,2,2}$ are $\Sub_{3,3,2}$ and $\Prank^{2,2,2}$ (which contains  $\Sub_{2,2,3}$).
Further, we'll see that $\Prank^{3,3,2}$ consists of four distinct components, namely the two subspace varieties $\Sub_{2,3,3}\cup \Sub_{3,2,3}$, the trifocal variety $X$, and $\Prank^{2,2,2}$.

\section{Symbolic computations using Representation Theory}\label{sec:rep}

In this section we compute the trifocal ideal $I(X)$ up to degree 9, and then find the minimal generators among those polynomials. 

A systematic algorithm to compute all polynomials in low degree  in the ideal of a $G$-variety is described in \cite{Landsberg-Manivel04}.  We carry out this algorithm for the trifocal variety.
In a word, the test is to decompose the ambient coordinate ring as a $G$-module and check every module of equations in low degree for membership via representation theory.

Recall that the trifocal variety $X$ is the closure of a $G$-orbit in $\PP V$, with $G=\SL(A)\times \SL(B)\times \SL(C)$.  This allows us to use tools from representation theory to compute and understand the ideal of $X$.

The coordinate ring of $\PP V$ has a $G$-module decomposition in each degree as
\[
S^{d}V^{*} = \bigoplus_{\lambda,\mu,\nu\vdash d} (S_{\lambda}A\otimes S_{\mu}B \otimes S_{\nu}C^{*})\otimes \CC^{m_{\lambda,\mu,\nu}}
,\]
where $S_{\lambda}A\otimes S_{\mu}B \otimes S_{\nu}C^{*}$ is an isotypic module associated to partitions $\lambda,\mu,\nu$ of $d$, and $m_{\lambda,\mu,\nu}$ is the multiplicity of that isotypic component, \cite[Proposition~4.1]{Landsberg-Manivel04}. One advantage to considering polynomials in modules is that we can compare different sets of polynomials no matter in what basis they are presented and determine if they are the same.

Even more useful is the following. Because we have a reductive group acting, we have the following splitting as $G$-modules
\[
k[V] = I(X) \oplus k[V]/I(X)
.\] To determine $I(X)$, then we just need to determine which irreducible $G$-modules are in $I(X)$. 

Here is a short synopsis of the algorithm in \cite{Landsberg-Manivel04}. Suppose $M$ is an irreducible $G$-module in $S^{\bullet}V^{*}$. To determine whether $M\subset I(X)$ or $M\subset k[V]/I(X)$, it suffices to check whether a random point on $X$ vanishes on the highest weight vector of $M$.  Random points of $X$  may be constructed by acting on a normal form by random elements of $G$.  If $M$ is an isotypic component of $S^{\bullet}V^{*}$ that occurs with multiplicity $m$, we first construct a basis of the highest weight space $\CC^{m}$ of $M$, by a straightforward construction involving Young symmetrizers. Then we select $m$ random points from $X$ and construct an $m\times m$ matrix  whose $i,j$-entry is the $i$th point evaluated on $j$th basis vector.  The kernel of this matrix tells the linear subspace of $M$ that is in the ideal of $X$.

Of course because we use random points, we should then re-verify all vanishing results symbolically to rule out false positives (there are no false negatives because non-vanishing at random points of $X$ implies non-vanishing on $X$.)  We did these extensive computations in Maple, and we have provided a sample of our code in the ancillary files accompanying the arXiv version of this article.

We can further determine which among the polynomial modules we find are actually minimal generators.
Suppose the degree $d$ piece of the ideal, $I(X)_{d}$, is known and has been input into Macaulay2 as \verb|I|.
  Then in the next step of the Landsberg-Manivel algorithm, for each new highest weight vector $f$ (a polynomial of degree $d+1$) check if $f \in \langle I(X)_{d}\rangle$ by quickly computing \verb|f%I|.
The module $\{G.f\}$ associated to the highest weight vector $f$ is in $\langle I(X)_{d}\rangle$ if and only if \verb|f%I| is zero.

We tabulate the results of this test applied to the trifocal variety below.
Again, to save space we 
write $S_{\lambda}S_{\mu}S_{\nu}$ in place of $S_{\lambda}A\otimes S_{\mu}B\otimes S_{\nu} C^{*}$.   The trifocal variety has an $\mathfrak{S}_{2}$ symmetry permuting the $A^{*}$ and $B^{*}$ factors.  To further save space, where appropriate, we will write $\mathfrak{S}_{2}.(S_{\lambda}S_{\mu})S_{\nu}$ to indicate the sum $S_{\lambda}S_{\mu}S_{\nu} + S_{\mu}S_{\lambda}S_{\nu}$.

\begin{prop}
Let $X$ denote the trifocal variety in $\PP V$ and let $M_{d}$ denote the space of minimal generators in degree $d$ of $I(X)$.  There are 10 minimal generators in degree 3, 81 in degree 5, and 1980 in degree 6.  The $G$-module structure of the minimal generators is as follows.
\[
\quad M_3 = \bw{3}\bw{3} S^{3},\]
\[\quad
M_{5}=(S_{221} S_{221})( S_{311} \oplus S_{221} )\]
\[M_{6}=
\left(\mathfrak{S}_{2}.(S_{222} S_{33}) ( S_{33} \oplus  S_{411})\right)
\oplus \mathfrak{S}_{2}.(S_{33} S_{321}) S_{321}
\]
\[
\oplus \left(\mathfrak{S}_{2}.(S_{33} S_{411})\oplus S_{33} S_{33}\right) S_{222}
,\]
and there are no other minimal generators in degree $\leq 9$.
\end{prop}
By recording the dimension of all modules that occur, this computes the first nine values of the Hilbert function of $k[V]/I(X)$:
\[
{27, 378, 3644, 27135, 166050, 865860, 3942162, 15966072, 58409126}
.\]

During our tests in Maple, we computed a basis of each module and provide these equations in the ancillary files accompanying the arXiv version of this paper.

We relate some of the polynomials we found to the known polynomials in \cite{Alzati-Tortora}.
Landsberg proves that $\Rank^{r}_{A}$ is the zero-set of $S^{r+1}A\otimes  \bw{r+1}B \otimes  \bw{r+1}C^{*} $ \cite[Proposition 7.2.2.2]{LandsbergTensorBook}. One can also phrase the condition that $T\in \Rank^{r}_{A}$ as the requirement that the matrix $T(A)$ of linear forms from $A$ has rank not exceeding $r$. If $A$ is $m$-dimensional, a basis of the module $S^{r+1}A\otimes  \bw{r+1}B \otimes  \bw{r+1}C $ is given as follows. Consider the slices $T^{1},\dots, T^{m}$ of $T$ in the $A$-direction, and dummy variables $x_{1},\dots,x_{m}$. The condition that $\rank(\sum_{i=1}^{m}x_{i}T^{i})\leq r$ is the condition that all coefficients (on the $x_{i}$) of the $(r+1)\times (r+1)$ minors of the matrix $\sum_{i=1}^{m}x_{i}T^{i}$ vanish. So a basis of $S^{r+1}A\otimes  \bw{r+1}B \otimes  \bw{r+1}C $ is given by the (polynomial) coefficients of these minors.

Recall that a normal form for a point $T$ on the trifocal variety has 
\[
T(C)
\cong
\begin{pmatrix}
0&c_{2}&c_{1}\\
c_{3}&0&0 \\
c_{1}&0&0
\end{pmatrix}.\]
And this matrix clearly has rank $\leq2$.

The above discussion implies that $X\subset \Rank^{2}_{C}$, and 
the module 
\[M_3:= \bw{3} A \otimes \bw{3} B\otimes S^3 C^{*}
\]
is in the trifocal ideal. 
These equations were also identified in \cite{Alzati-Tortora}.

We now describe the two modules in $M_{5}$.
The highest weight vectors are polynomials with (respectively) 104 and 244 monomials and multi degrees [(2,2,1),(2,2,1),(3,1,1)]  and 
[(2,2,1),(2,2,1),(2,2,1)], in the ring
\[
k[a_{11},\dots, a_{33}, b_{11},\dots, b_{33}, c_{11},\dots, c_{33}].
\]
Here we are using $a_{ij}$ to denote $T_{ij1}$, $b_{ij}=T_{ij2}$, and $c_{ij}=T_{ij3}$.
Typical terms of the highest weight vectors are
\begin{multline*}
\dots -b_{31}c_{22}a_{13}b_{12}a_{11}+b_{31}a_{12}^2 b_{23}c_{11}-b_{31}c_{21}a_{12}^2b_{13}\\
+b_{22}c_{31}a_{13}a_{12}b_{11}-b_{22}c_{31}a_{11}b_{13}a_{12}-a_{32}c_{22}b_{11}^2a_{13}\dots
\end{multline*}
and
\begin{multline*}\dots
-a_{21}^2b_{12}b_{33}c_{12}-2a_{12}a_{33}b_{11}b_{22}c_{21}-a_{21}^2b_{12}b_{13}c_{32}
\\+a_{12}a_{23}b_{11}b_{21}c_{32}+a_{21}^2b_{12}^2c_{33} \dots,
\end{multline*}
respectively.
The basis of $S_{221}S_{221}S_{221}$ consists of 27 polynomials which are all equal after a change of indices. All the coefficients come from the set $\{-5, -2, -1, 1, 2, 4\}$.
The basis of $S_{221}S_{221}S_{311}$ consists of 54 polynomials which are of two different types having either 104 or 64 monomials and coefficients in the set $\{-1,1\}$. The polynomials themselves can be downloaded from the web as mentioned above.

For $M_6$ we can give a similar description.
$S_{222}S_{33}S_{33}$ and $S_{222}S_{33}S_{411}$ and $S_{33}S_{222}S_{411}$ are all 100-dimensional, each with a basis consisting of polynomials that have between 66 and 666 monomials and small (absolute value no greater than $4$) integer coefficients.
Similarly
$S_{321}S_{33}S_{321}$ is 640-dimensional with a basis consisting of polynomials that have between 60 and 732 monomials and small integer coefficients.
The full set of polynomials is available with the ancillary files on the arXiv version of the paper.

After computing  $I(X)_{d}$ for small $d$, we computed a Gr\"obner basis of $J = \langle M_{3}+M_{5}+M_{6}\rangle$ in Macaulay2.  Surprisingly, this computation finished in a few minutes --- it actually took longer to load the polynomials into M2 than it took to compute the Gr\"obner basis. We were also able to compute a Gr\"obner basis of $\mathfrak{S}_{3}.M_{3}=\left(\bw{3}\bw{3} S^{3}\right)\oplus \left( \bw{3} S^{3}\bw{3}\right)\oplus  \left(S^{3}\bw{3}\bw{3}\right)$, the 30 cubic equations defining the rank variety $\Prank^{2,2,2}$.  We record the results of these computations:

\begin{prop}\label{prop:m2} 
\label{prop:low-deg-computations} Let $X$ denote the trifocal variety and let $M_{d}$ denote the space of minimal generators in $I(X)_{d}$.
The following computations done over $\mathbb{Q}$ hold:
\[\deg(\V(\mathfrak{S}_{3}.M_{3})) = 1035 \text{ and }\codim(\V(\mathfrak{S}_{3}.M_{3})) = 10.\]
\[\deg(\V(M_{3}+M_{5}+M_{6})) = 297 \text{ and } \codim(\V(M_{3}+M_{5}+M_{6})) = 8.\]
\end{prop}
\begin{proof}The proof is by computations in Macaulay2 \cite{M2} that we provide with the ancillary files in the arXiv version of the paper.\end{proof}

Though we don't need it for our proof, we were also able to compute the Hilbert polynomial of $J=\langle M_{3}+M_{5}+M_{6}\rangle$:
\begin{multline*}
69 {P}_{5}-423 {P}_{6}+882 {P}_{7}-204 {P}_{8}-2565 {P}_{9}
\\
+5751{P}_{10}-6129 {P}_{11}+3402 {P}_{12}-783 {P}_{13}+100 {P}_{14}
\\
-525{P}_{15}+1038 {P}_{16}-909 {P}_{17}+297 {P}_{18},
\end{multline*}
where we use the variables $P_{i}$ following the standard normalization used in Macaulay2 to describe the Hilbert Polynomial.

\section{Numerical Algebraic Geometry: Bertini}\label{sec:Bertini}
In Numerical algebraic geometry, and specifically using the program Bertini, one can compute numerical primary decompositions of ideals if the number of equations and the degrees of those equations are relatively small. 
 In contrast to Gr\"obner basis computations where typically more equations is better, in numerical algebraic geometry it is better to start with the lowest degree and lowest number of equations that one can understand. Then one can try to compute a numerical primary decomposition and attempt to work by other means to obtain a geometric description of the components indicated by Bertini.

Following this philosophy, we started with the 10 equations in degree 3 given by the complete vanishing of
\[
\det(x_{1}Z_{1} +x_{2}Z_{2} + x_{3}Z_{3})
,\]
which define $\Rank_{C}^{2}$.
Recall that $(Z_1\mid Z_2\mid Z_3)$ is the flattening of a tensor in the $C$-direction.  These are specifically the 10 equations defining the module $M_3$; that is, a basis for $M_3$.

After about 6.5 hours of computational time on 2 processors (and some help from J. Hauenstein getting the initial parameters right), or just under 10 minutes on Hauenstein's cluster, Bertini succeeded to compute the following numerical decomposition.
\begin{computation}\label{comp:Bertini}
Let $M_{3}$ denote the 10 coefficients (in $x_{1},x_{2},x_{3}$) of the cubic $\det(x_{1}Z_{1} +x_{2}Z_{2} + x_{3}Z_{3})$.
Up to the numerical precision of Bertini, the zero set of $M_{3}$ has precisely 4 components:

In codimension 7 there are 2 components, each of degree 36.

In codimension 8 there is 1 component of degree 297.

In codimension 10 there is 1 component of degree 1035.
\end{computation}
It is not too hard to show that the two components in codim 7 are the subspace varieties
\[\Sub_{3,2,3} \cup \Sub_{2,3,3}\] 

This is because they have the correct dimension, $M_3$ is in both ideals, and their ideals are generated (respectively) by
\[(\bw{3}\bw{3}S^{3}) \oplus (S^{3}\bw{3}\bw{3})\qquad \text{ and}\qquad (\bw{3}\bw{3}S^{3}) \oplus (\bw{3}S^{3}\bw{3}).\]

We know that the trifocal tensor variety is in the zero set and has codimension 8. It is not contained in either subspace variety, so we may conclude that $X$ corresponds to the codimension 8 component in the numerical decomposition.  We also learn that $X$ has degree 297.

The variety $\Prank^{2,2,2}$ must correspond to the codimension 10 component, which we prove in Proposition \ref{prop:F} below. In addition, this Bertini computation also tells that $\Prank^{2,2,2}$ has degree 1035.

\section{Nurmiev's classification of orbit closures and the proof of the main theorem}\label{sec:Nurmiev}
The orbits of $\SL(3)^{\times 3}$ acting on $\CC^{3}\otimes \CC^{3}\otimes \CC^{3}$ have been classified by Nurmiev \cite{Nurmiev}, who also computed the closure of most orbits.

Using Nurmiev's list of normal forms, we can quickly check which orbits are contained in $\V(M_{3})$ by taking a parameterized representative for each orbit (normal form) and evaluating the polynomials in $M_3$ on that representative.  This can be carried out in a straightforward manner in any computer algebra system.

The following orbits from Nurmiev's list of nilpotent orbits \cite[Table 4]{Nurmiev2} are in $\V(M_{3})$:
$9$, $9'$, $11''$, $12$, $12'$, $13$, $13'$, $14$, $15$, $15'$, $16$, $16'$, $17$, $17'$, $17'',$
$18$, $18'$, $18''$, $19$, $19'$, $19''$, $20\ (=20'=20'')$, $21$, $21'$, $21''$, $22$, $22'$, $23$, $23'$, $23'',$
$24\ (=24'=24'')$, $25\ (=\emptyset)$,
see \cite{Nurmiev} for an explanation of the notation used.

After considering the nilpotent orbits, we must also consider the semi-simple orbits together with their nilpotent parts.  Among these orbits, our tests found that only one semi-simple orbit, namely the one corresponding to Nurmiev's fourth family, is in  $\V(M_{3})$. In our notation we may represent this normal form as
\begin{multline*}
F=\lambda(a_1\otimes b_2\otimes c_3+a_2\otimes b_3\otimes c_1+a_3\otimes b_1\otimes c_2\\
-a_1\otimes b_3\otimes c_2-a_2\otimes b_1\otimes c_3-a_3\otimes b_2\otimes c_1),\end{multline*}
for any scalar $\lambda\neq 0$, but over the complex numbers this scalar may be absorbed.

\begin{remark}
As a matrix with linear entries either in $A^{*}$, $B^{*}$ or $C$ a normal form for $F$ is always of the form
\[
\begin{pmatrix}
0 & x & -y \\ -x& 0 & z \\ y&-z &0
\end{pmatrix}
.\]
Namely this orbit corresponds to the skew symmetric matrices. 
Moreover, since this matrix is skew-symmetric, it always has even rank, and thus we find that $F\in \Prank^{2,2,2}$ with no computation necessary.
\end{remark}

Next we consider the closures of all the nilpotent orbits in \cite[Table 4]{Nurmiev2}.  Our inclusion poset diagram in Figure \ref{fig:graph} is enlightening. For all arrows except for those emanating from $F$, the diagram is a restatement of results in \cite{Nurmiev2}.  Namely, if orbit $Q$ is in the closure of orbit $P$ (as indicated by Nurmiev's table) and there isn't already a directed path from $P$ to $Q$ we draw an arrow from $P$ to $Q$.
\begin{figure}\caption{A poset diagram for orbit closures in $\Rank_{C}^{2}$}\label{fig:graph}
\includegraphics[scale=0.75]{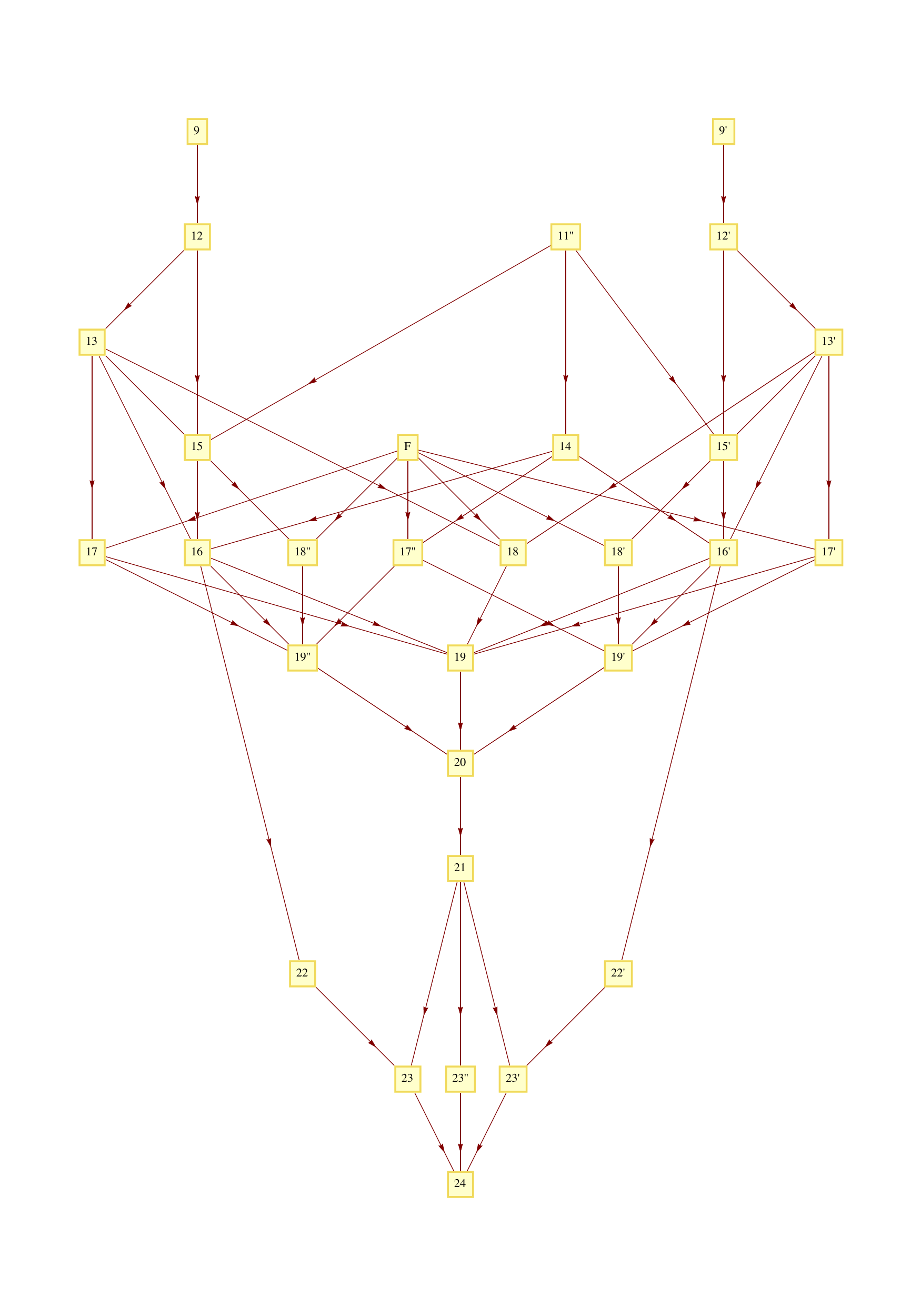}
\end{figure}

We then consider all orbits that are in the zero set of $M_{3}$ for all permutations of $A$,$B$ and $C$.  We write this zero set in shorthand as  $\V(\mathfrak{S}_{3}.M_{3})$.  Since $F$ is a zero of $\mathfrak{S}_{3}.M_{3}$, every orbit in its closure must also be in this zero set. A straightforward computation shows that these orbits in $\V(\mathfrak{S}_3.M_3)$ are numbers 17-21, 23, 24 (and all of their primed versions).  Of course this does not imply that these orbits are actually in the closure of $F$. However, it is enough to check that orbits 17, 18 (and their primed versions) are in the closure of $F$.

\begin{prop}\label{prop:F}
The zero set $\V(\mathfrak{S}_{3}.M_{3})$ (which equals $\Prank^{2,2,2}$ by definition) is irreducible and is the closure of the orbit $F$ above.

Moreover, the orbit associated to the normal form $F$ is not contained in any of the other orbit closures in Figure \ref{fig:graph}\end{prop}
\begin{proof}
This proof is entirely computational, but since we did not find it in the literature, we provide the computations here.

Orbit closures consist of one orbit of the top dimension along with other orbits of lower dimension.
So we conclude by counting dimensions that none of the orbits $9$, $9'$, $11''$, $12$, $12'$, $13$, $13'$, $14$, $15$ or $15'$  are in the closure of $F$.  Later we will show that $17$, $18$ and all their primed versions are in the closure of $F$.

We claim that none of the orbits of higher dimension ($9$, $9'$, $11''$, $12$, $12'$, $13$, $13'$) contain $F$ in their closure.
Consider the normal form of $F$
\[
\begin{pmatrix}
0 & x & -y \\ -x& 0 & z \\ y&-z &0
\end{pmatrix}
,\]
and flatten it to the matrix
\[
\left(
\begin{matrix}
0 & x & 0 \\ -x& 0 & 0 \\ 0&0 &0
\end{matrix}
\left\vert\vphantom{\begin{matrix}\\ \\ \end{matrix}}\right.
\begin{matrix}
0 & 0 & -y \\ 0& 0 & 0 \\ y&0 &0
\end{matrix}
\left\vert\vphantom{\begin{matrix}\\ \\ \end{matrix}}\right.
\begin{matrix}
0 & 0 & 0 \\ 0& 0 & z \\ 0&-z &0
\end{matrix}
\right)
,\]
which, for general choices of $x,y,z$ has full rank.  The other slices have similar format, and this shows that $F$ is not contained in either of the subspace varieties $\Sub_{2,3,3}$ or $\Sub_{3,2,3}$ (the closures of orbits $9$ and $9''$). This also implies that $F$ is not in the closure of any orbit contained in the closure of $9$ or $9''$.

To show that $F$ is not contained in $X$, we could demonstrate a polynomial in $I(X)$ that does not vanish on $F$.  We already noted that the skew-symmetric matrices in $F$ has rank 2 or less and thus vanishes on all polynomials in $M_{3}$. The ideal $I(X)$ has no minimal generators in degree 4, so we must start to consider polynomials in degree 5 or higher. Here we notice by direct computation that neither of the modules in $M_{5}$ vanish on $F$, separating $F$ from $X$.

Another way to conclude $\Rank_{C}^{2}\not\subset \V(M_{5})$, without computation, is to consider the degree 5 piece of the ideal generated by $M_{3} = \bw{3}\bw{3}S_{3}$. The Pieri formula gives that every module of $I(M_{3})$ in degree 5 must have a partition in the first position whose first part is at least 3. On the other hand, the module $S_{221}S_{221}S_{221}$ in $M_{5}$ fails this property.  So it cannot be in the ideal $I(M_{3})$ of $\Rank_{C}^{2}$, and thus $\Rank_{C}^{2}\not\subset X$.

Any orbit in $\V(M_{3})$ is either in the closure of $F$ or in the closure of $9$, $9'$ or $11''$.  Moreover, because $F$ is not in the closure of $9$, $9'$ or $11''$,  its closure must be an irreducible component of $\V(M_{3})$.

Since $\mathfrak{S}_{3}.M_{3}$ manifestly has $\mathfrak{S}_{3}$ symmetry, so does its zero-set. Thus to prove that $\V(\mathfrak{S}_{3}.M_{3})$ is irreducible, we need to show that orbits 17, 18 and their primed versions are contained in the closure of $F$.  This will imply irreducibility of $\V(\mathfrak{S}_{3}.M_{3})$ because every orbit contained there is in the closure of a single orbit.

Since $F$ is symmetric with respect to permutation by $\mathfrak{S}_{3}$, it suffices to prove that orbits 17 and 18 are contained in the closure of $F$. To do this we exhibit a sequence of points in the orbit of $F$ that converge to each orbit.  We prefer to work with the normal forms both as tensors and as matrices of linear forms. The group operations allowed are row and column operations as well as general linear changes of coordinates on each linear form appearing. 

Nurmiev's orbit 17 has normal form given by the matrix
\[
\begin{pmatrix}
0 & a_{1} &a_{2} \\
a_{1}&0&0\\
a_{2}&0&0\\
\end{pmatrix}
,\]
or the tensor
\[ 
a_{1}\otimes (b_{1}\otimes c_{2} + b_{2}\otimes c_{1}) + 
a_{2}\otimes (b_{1}\otimes c_{3} + b_{3}\otimes c_{1})
.\]
Replacing $b_{1}$ with $-b_{1}$ we obtain the tensor
\[ 
a_{1}\otimes (-b_{1}\otimes c_{2} + b_{2}\otimes c_{1}) + 
a_{2}\otimes (-b_{1}\otimes c_{3} + b_{3}\otimes c_{1})
\]\[=
-a_{1}\otimes b_{1}\otimes c_{2} +a_{1}\otimes b_{2}\otimes c_{1}) + 
-a_{2}\otimes b_{1}\otimes c_{3} +a_{2}\otimes b_{3}\otimes c_{1})
,\]
which corresponds to the (skew symmetric) matrix
\[
\begin{pmatrix}
0 & -a_{1} &-a_{2} \\
a_{1}&0&0\\
a_{2}&0&0\\
\end{pmatrix}
.\]
This matrix is a limit of the following matrices in the orbit of $F$
\[
\begin{pmatrix}
0 & -a_{1} & -a_{2} \\ a_{1}& 0 & z_{n} \\ a_{2}&-z_{n} &0
\end{pmatrix}
,\]
where $z_{n}\to 0$. 

For the orbit 18, we consider the following representative as a tensor
\[
a_{1}\otimes (b_{1}\otimes c_{1} + b_{2}\otimes c_{2}) + 
a_{2}\otimes (b_{1}\otimes c_{2} + b_{2}\otimes c_{3})  
\]
which corresponds to the matrix
\[
\begin{pmatrix}
a_{1} &a_{2} &0 \\
0&a_{1} &a_{2} \\
0&0&0
\end{pmatrix}
.\]
By cycling rows 1, 2 and 3 we at least have a matrix with $0$ diagonal,
\[
\begin{pmatrix}
0&a_{1} &a_{2} \\
0&0&0\\
a_{1} &a_{2} &0 
\end{pmatrix}
.\]
Now we work with the following matrix normal form for $F$:
\[
\begin{pmatrix}
0 & x & -y \\ -x& 0 & z \\ y&-z &0
,\end{pmatrix}\]
which corresponds to the tensor 
\[
F=z\otimes b_2\otimes c_3+y\otimes b_3\otimes c_1+x\otimes b_1\otimes c_2
-z\otimes b_3\otimes c_2-y\otimes b_1\otimes c_3-x\otimes b_2\otimes c_1.\]
First we set $x = a_{1}$ and $-y=a_{2}$. Considering a limit of such tensors gives a tensor in the closure with $b_{2}=0$ 
corresponding to the matrix
\[
\begin{pmatrix}
0 & a_{1} & a_{2} \\ 0& 0 & 0 \\ -a_{2}&-z &0
\end{pmatrix}.\]
Multiply the 3rd row by $\frac{-a_{1}}{a_{2}}$ to get another matrix in the same orbit
\[
\begin{pmatrix}
0 & a_{1} & a_{2} \\ 0& 0 & 0 \\ a_{1}&-z\frac{-a_{1}}{a_{2}} &0
\end{pmatrix}.\]
Finally we set $z = \frac{a_{2}^{2}}{a_{1}}$ to yield a matrix in the orbit 18.
\end{proof}

The above discussion provides the following effective test for a given tensor $T$ to be a trifocal tensor. Namely the orbit of trifocal tensors is precisely the $G$-invariant set of tensors with $\Prank(T) = (3,3,2)$ or some permutation thereof, and $\Frank(T) = (3,3,3)$. These two conditions contain both vanishing and non-vanishing conditions.  We phrase this as an algorithm to determine whether a given tensor is a trifocal tensor.

\begin{algor0}
\hfill\par
\textbf{Input:} A tensor $T\in \CC^{3}\otimes \CC^{3}\otimes \CC^{3}$.
\begin{itemize}
\item Replace $T$ by a change of coordinates (either arbitrary or random) from $GL(A)\times GL(B) \times GL(C)$ applied to $T$.
\item Compute the projections $T(A),T(B),T(C)$. 
Is $\Prank(T)=(3,3,2)$ (or some permutation) and no less? 
\subitem NO: stop, $T$ is not a trifocal tensor. 
\subitem YES: continue 
\item Compute all 3 flattenings. Is $\Frank(T)=(3,3,3)$ and no less?
\subitem NO: stop, $T$ is not a trifocal tensor.
\subitem YES: $T$ is a trifocal tensor.
\end{itemize}
\par
\end{algor0}
If one uses arbitrary changes of coordinates (with parameters) the conclusions of Algorithm 1 hold without modification.  However, it may be difficult to perform the tests. If one uses random changes of coordinates, Algorithm 1 will go quickly, and the negative conclusions are sure, but the positive conclusions will hold only with high probability.

This test is effective because it involves computing the ranks of three $3\times 3$ matrices and the ranks of three $3\times 9$ matrices.  This test is similar in spirit to the results in \cite[Section 4]{Alzati-Tortora}.

Proposition \ref{prop:F} yields the following geometric statement.
\begin{prop}\label{prop:decomp} Let $X$ denote the trifocal variety (the closure of orbit $11''$).  Then the irreducible decomposition of $\V(M_{3})=\Prank^{3,3,2}$ is
\[
\V(M_{3})= \Sub_{2,3,3} \cup \Sub_{3,2,3} \cup X \cup \Prank^{2,2,2}.
\]
\end{prop}
\begin{proof}
To see that $\V(M_{3})$ contains the four listed components, just construct a normal form for each and notice that the associated matrix in linear forms  has rank $<3$.  To see that these are the only components, look at the orbit closure diagram in Figure \ref{fig:graph}, which displays all orbits in $\V(M_{3})$ and is justified by \cite{Nurmiev2} and Proposition \ref{prop:F}. Notice that there are 4 sources in the directed graph representing the poset and these correspond to the only irreducible components in the decomposition. 
\end{proof}

Nurmiev's table includes the dimension of the stabilizer of each orbit, which tells the codimension of each of the components: $\codim(\Sub_{2,3,3}) = 7$, $\codim(X) = 8$, $\codim(\Prank^{2,2,2})  = 10$.  Nurmiev's computation is confirmed by the computation done in Bertini; however, Bertini tells us a bit more, namely the degree of each component.
\begin{prop}\label{prop:irred}
The zero set $\V(M_{3}+M_{5}+M_{6})$ is irreducible and agrees with $X$ set-theoretically.
\end{prop}
\begin{proof}
We need to show that when we intersect $\V(M_{3})$ with $\V(M_{5}+M_{6})$ that all of the orbits that remain are actually in the trifocal variety.  It suffices to show that orbits 17, 17', 18', 18'' are not in $\V(M_{5}+M_{6})$. This is because by considering the orbit closure poset diagram  in Figure \ref{fig:graph}, these orbits are contained in all other orbits in $\V(M_{3})$ that are not in the trifocal variety $X$, so if they are not in $\V(M_{5}+M_{6})$, then no other $G$-orbit in $\V(M_{3})$ outside of $X$ is in $\V(M_{3}+M_{5}+M_{6})$.

By direct computation, we find that the module $S_{33}S_{222}S_{411}$ does not vanish on orbit 17, 
the module $S_{222}S_{33}S_{411}$ does not vanish on orbit 17',
the module $S_{33}S_{33}S_{222}$ does not vanish on 18',
and $S_{222}S_{33}S_{33}$ does not vanish on 18''.  On the other hand, each of these modules are in $M_{6}$.
\end{proof}

\section{The ideal $J$ is prime}\label{sec:steven}
Let $J = \langle M_{3}+M_{5} + M_{6}\rangle$, where $M_{d}$ are the minimal generators of $I(X)$ in degree $d$.
The trifocal variety $X$ is irreducible because it is a parameterized variety. The fact that the zero set $V(J)$ is irreducible and equals $X$ is the content of Proposition \ref{prop:irred}. So $I(X)$ and $J$ agree up to radical. It remains to check that there are no embedded components.

The classification of $G$-orbits in $V$ also yields a classification of minimal $G$-invariant prime ideals. To every orbit is associated the prime ideal of its orbit closure. 
\begin{remark}
Here we also use the fact that if $G$ is a connected group, and $J$ is a $G$-stable ideal, then the minimal primes in any primary decomposition of $J$ are $G$-stable. 
This essentially follows from the fact that if $J =\cap_{i}Q_{i}$ is a primary decomposition with primary ideals $Q_{i}$ associated to primes $P_{i}$, then $gJ = J = \cap_{i}gQ_{i}$ for any $g \in G$ and this action must permute the $P_{i}$ by the uniqueness of minimal primes. But since $G$ is connected, this permutation must be trivial.
\end{remark}

The poset in Figure \ref{fig:graph} shows that the minimal prime ideals that contain $I(X)$ are those corresponding to orbits 14, 15, and $15'$. Let $P_{14}$, $P_{15}$ and $P_{15'}$ denote the corresponding prime ideals.
Then we must have $J\subset P_{14}\cap P_{15} \cap P_{15'}$.  On the other hand, we know that $\sqrt J$ is prime and equals $I(X)$. 
So a primary decomposition of $J$ is of the form
 $J= I(X)\cap Q_{14} \cap Q_{15} \cap Q_{15}$, for some primary ideals $Q_{i}$ associated to the primes $P_{i}$.  We will show that the multiplicity of each $Q_{i}$ with respect to $P_{i}$ is zero. 

If we show this, we don't have to consider possible embedded components coming from the other orbits in the closure of $X$ because these ideals contain $P_{14}$, $ P_{15}$ and $P_{15'}$.  Moreover, since $X$ and $J$ have an $\mathfrak{S}_{2}$ symmetry, if we show that the $P_{15}$ does not occur in the primary decomposition, then neither does $P_{15'}$.

We will use a basic fact from commutative algebra. We found \cite[Theorem 12.1]{BrunsVetter} a useful formulation for understanding this type of test.

\begin{prop}\cite[Proposition~4.7]{AtiyahMacDonald}
Let $\mathfrak{a}$ be a decomposable ideal in a ring $A$, let $\mathfrak{a} =\cap_{i-1}^{n}\mathfrak{q}_{i}$ be a minimal primary decomposition and let $\mathfrak{p}_{i}$ be the prime ideal associated to the primary ideal $\mathfrak{q_{i}}$.  Then
\[
\cup_{i=1}^{n}\mathfrak{p}_{i} = \{x\in A\mid (\mathfrak{a}:x) \neq \mathfrak{a}  \}.
\]
In particular, if the zero ideal is decomposable, the set $D$ of zero-divisors of $A$ is the union of the prime ideals belonging to $0$.
\end{prop}

We also have the following well-known fact (see for example \cite{Eisenbud}).
\begin{prop} Let $R= k[x_{0},\dots x_{n}]$, let $J$ be an ideal in $R$ and  suppose $f\in R$ has degree $d$ and is not a zero-divisor in $R/J$. 
Then we have the following identity of Hilbert series:
$$(1-t^{d})H_{R/J}(t) = H_{R/(J+f)}(t).$$
\end{prop}
\begin{proof}
This is completely standard, but we recall the proof here for the reader's convenience and because it elucidates the ideas we will use later.

If $f$ is not a zero divisor, the following sequence is exact
\begin{equation}\label{eq:seq}
0 \rig{} (R/J)(-d) \rig{f} R/J \rig{} R/(J+f) \rig{} 0
.\end{equation}
Since $H_{(R/J)(-d)}(t) = t^{d}H_{(R/J)}(t)$, the result follows from the additivity of Hilbert series.
\end{proof}
\begin{remark}
If $f$ is actually a zero-divisor of $R/J$, then in some degree $t'$, the graded version of the sequence \eqref{eq:seq} will have a kernel $K$ larger than expected. This will force the inequality in 
\[t'^{d}H_{(R/J)}(t')=H_{(R/J)(-d)}(t') < H_{K}(t').\] 
In this case we will have 
\[t'^{d}H_{(R/J)}(t') - H_{R/J}(t')  <H_{K}(t') - H_{R/J}(t')= H_{R/(J+f)}(t'),\]
which implies that
\[(1-t'^{d})H_{R/J}(t') < H_{R/(J+f)}(t').\]

\end{remark}
The previous results allow for the following test.  Since the zero-divisors of $J$ correspond to the union of prime ideals  $P_{i}$ that contain $J$, we can select one $f \in P_{i}$ of degree $d$ which vanishes on the subvariety $\V(P_{i})\subset \V(J)$ but does not vanish on $\V(J)$.  If we show that  $(1-t^{d})H_{R/J}(t) = H_{R/(J+f)}(t)$, then $f$ is not a zero-divisor of $R/J$. This would show that $P_{i}$ could not have been a prime ideal associated to $J$.

We provide the results of this computational test with the prime ideals $P_{15}$ and $P_{14}$. We wanted to check if a map had a kernel, so we worked over characteristic 101. Non-vanishing modulo a prime $p$ implies non-vanishing in characteristic 0. 

In Macaulay 2 we computed a Gr\"obner basis of $J=\langle M_{3} + M_{5} + M_{6}\rangle$ in about 30 seconds.
The Poincare polynomial $P_{J}$ of $R/J$ is
\begin{multline*}
P_{J}=1-10 T^{3}-81 T^{5}-1605 T^{6}+18117 T^{7}-77517 T^{8}+192794 T^{9}-315792 T^{10}
\\
+350676 T^{11}-243572 T^{12}+48438 T^{13}+116883 T^{14}-175239 T^{15}+140238 T^{16}\\
-75330 T^{17}+27954 T^{18}-6912 T^{19}+1026 T^{20}-69 T^{21}
.\end{multline*}

For the prime ideal $P_{15}$ we constructed slices in the $B$-direction $Y_{1},Y_{2},Y_{3}$, computed  $\det(x_{1}Y_{1} + x_{2}Y_{2} + x_{3}Y_{3}) \equiv0$ and selected the polynomial $f$ as the coefficient of $x_{1}^{3}$. Precisely, 
\[
f =\det\bgroup\begin{pmatrix}a_{11}&
     a_{12}&
     a_{13}\\
     b_{11}&
     b_{12}&
     b_{13}\\
     c_{11}&
     c_{12}&
     c_{13}\\
     \end{pmatrix}\egroup
.\]
This polynomial $f$ vanishes on $\Prank^{3,2,3}$ and $\V(P_{15})$ but not on $\Prank^{3,3,2}$, and thus not on $X$. 

Computing a Gr\"obner basis of $J +\langle f \rangle$ took about 10 hours on a server that allowed us to use 16GB of RAM and up to 8 Intel(R) Xeon(R) CPU           X5460 3.16GHz processors.  The Poincare polynomial $P_{f}$ of $R/(J+f)$ is
\begin{multline*}P_{f}=
1-11t^3-81t^5-1595t^6+18117t^7-77436t^8+194399t^9\\
-333909t^{10}+428193t^{11}-436366t^{12}+364230t^{13}-233793t^{14}\\
+68333t^{15}+91800t^{16}-192213t^{17}+203193t^{18}-147150t^{19} \\
+76356t^{20}-28023t^{21}+6912t^{22}-1026t^{23}+69t^{24}
.\end{multline*}
Now it is easy to check that $(1-t^{3})P_{J} = P_{f}$, which implies \[(1-t^{3})H_{R/J}(t) = H_{R/(J+f)}(t),\] and thus $f$ is not a zero-divisor of $R/J$. The prime $P_{15}$ is thus not an embedded prime of $J$.  By the $\mathfrak{S}_{2}$ symmetry of $J$, we conclude that $P_{15'}$ is also not an embedded prime of $J$.

For the prime ideal $P_{14}$, the module $S_{22}S_{211}S_{211}$ vanishes on $\V(P_{14})$ but not on $X$. We select the highest weight polynomial $g$ for our test:
\begin{multline*}
g =a_{13}a_{21}b_{12}c_{21}-a_{13}a_{21}c_{12}b_{21}+3a_{23}a_{12}b_{11}c_{21}+c_{13}a_{21}a_{12}b_{21}- \\
b_{13}a_{21}a_{12}c_{21}-3a_{23}a_{12}c_{11}b_{21}+a_{23}a_{11}b_{22}c_{11}-a_{23}a_{11}c_{22}b_{11}+\\
2c_{22}a_{21}a_{13}b_{11}-c_{22}a_{13}a_{11}b_{21}-2b_{22}a_{21}a_{13}c_{11}+b_{22}a_{21}c_{13}a_{11}+\\
b_{22}a_{13}a_{11}c_{21}-c_{22}a_{21}b_{13}a_{11}+c_{23}a_{11}a_{22}b_{11}+2b_{23}a_{21}a_{12}c_{11}-\\
b_{23}a_{21}c_{12}a_{11}-b_{23}a_{12}a_{11}c_{21}-c_{13}b_{12}a_{21}^2+b_{13}c_{12}a_{21}^2-\\
c_{23}b_{22}a_{11}^2+b_{23}c_{22}a_{11}^2-a_{22}a_{21}b_{13}c_{11}-2c_{23}a_{21}a_{12}b_{11}+\\
c_{23}a_{21}b_{12}a_{11}+c_{23}a_{12}a_{11}b_{21}-b_{23}a_{11}a_{22}c_{11}+a_{23}a_{21}b_{12}c_{11}-\\
a_{23}a_{21}c_{12}b_{11}+a_{22}a_{21}c_{13}b_{11}-3a_{22}a_{13}b_{11}c_{21}+3a_{22}a_{13}c_{11}b_{21}-\\
2c_{21}a_{23}b_{12}a_{11}+2c_{21}a_{22}b_{13}a_{11}+2b_{21}a_{23}c_{12}a_{11}-2b_{21}a_{22}c_{13}a_{11}.
\end{multline*}

Computing a Gr\"obner basis of $J+\langle g\rangle$ took about 45 hours to finish on a server that allowed us to use 16GB of RAM and up to 8 processors. The Poincare polynomial $P_{g}$ of $R/(J+g)$ is
\begin{multline*}P_{g}=
1-10t^3-t^4-81t^5-1605t^6+18127t^7-77517t^8+192875t^9\\
-314187t^{10}+332559t^{11}-166055t^{12}-144356t^{13}+432675t^{14}\\
-525915t^{15}+383810t^{16}-123768t^{17}-88929t^{18}+168327t^{19}\\
-139212t^{20}+75261t^{21}-27954t^{22}+6912t^{23}-1026t^{24}+69t^{25}
.\end{multline*}
It is again a simple check that $(1-t^{4})P_{J} = P_{g}$, which implies 
\[(1-t^{4})H_{R/J}(t) = H_{R/(J+g)}(t).\]
As before, $g$ is not a zero-divisor of $R/J$, and the prime $P_{14}$ is not an embedded prime of $J$. 

We have shown the following
\begin{theorem}\label{thm:prime}
The ideal $J = \langle M_{3}+M_{5}+M_{6}\rangle$ is prime.
\end{theorem}
\begin{proof}
By Proposition \ref{prop:irred} we know that $\sqrt J =I(X)$. By Proposition \ref{prop:m2}, we know that the degree of the top dimensional component of $J$ is $297$, counted with multiplicity.
By Computation \ref{comp:Bertini}, we know that the degree of $X$ is 297. So we know that in a primary decomposition of $J$, $I(X)$ occurs with multiplicity 1. It only remains to rule out embedded primes. By the above discussion, if we have a primary decomposition of the form $J = I(X) \cap Q_{14}\cap Q_{15} \cap Q_{15'}$, where the $Q_{i}$ are primary ideals associated to the prime ideals $P_{i}$, then we showed that their multiplicity must be zero. So $J= I(X)$, and in particular $J$ is prime.
\end{proof}
This completes the proof of Theorem \ref{thm:main}.  We conjecture that a similar calculation will work to show that the ideal of the orbit closure associated to $F$ is minimally generated by $\mathfrak{S}_{3}.M_{3}$.

\section*{Acknowledgements}
We would like to thank Bernd Sturmfels for suggesting this problem to us, along with some suggestions for a few of the computations.

We would also like to thank Jonathan Hauenstein for his expertise in Numerical Algebraic Geometry and Bertini, which provided an unquestionably useful component to our studies into the trifocal ideal.  We also thank Steven Sam for suggesting we try the crucial computation in Section \ref{sec:steven}. \bibliographystyle{amsalpha}
\bibliography{trifocal_bib}

\end{document}